\documentclass[a4paper,12pt]{article}

\usepackage[textwidth=125mm, textheight=195mm]{geometry}
\usepackage[english]{babel}
\usepackage{graphicx}
\usepackage{amsmath}
\usepackage{amsfonts}
\usepackage{amsthm}
\usepackage{epstopdf}
\usepackage{color}

\begin{document}

\newcommand{\wk}{\mbox{$\,<$\hspace{-5pt}\footnotesize )$\,$}}

\numberwithin{equation}{section}
\newtheorem{teo}{Theorem}
\newtheorem{lemma}{Lemma}

\newtheorem{coro}{Corollary}
\newtheorem{prop}{Proposition}
\theoremstyle{definition}
\newtheorem{definition}{Definition}
\theoremstyle{remark}
\newtheorem{remark}{Remark}

\newtheorem{scho}{Scholium}
\newtheorem{open}{Question}
\newtheorem{example}{Example}
\numberwithin{example}{section}
\numberwithin{lemma}{section}
\numberwithin{prop}{section}
\numberwithin{teo}{section}
\numberwithin{definition}{section}
\numberwithin{coro}{section}
\numberwithin{figure}{section}
\numberwithin{remark}{section}
\numberwithin{scho}{section}

\bibliographystyle{abbrv}

\title{Duality of gauges and symplectic forms in vector spaces}
\date{}

\author{Vitor Balestro\footnote{Corresponding author}  \\ Instituto de Matem\'{a}tica e Estat\'{i}stica \\ Universidade Federal Fluminense \\ 24210201 Niter\'{o}i \\ Brazil \\ vitorbalestro@id.uff.br \and Horst Martini \\ Fakult\"{a}t f\"{u}r Mathematik \\ Technische Universit\"{a}t Chemnitz \\ 09107 Chemnitz\\ Germany \\ martini@mathematik.tu-chemnitz.de \and Ralph Teixeira \\ Instituto de Matem\'{a}tica e Estat\'{i}stica \\ Universidade Federal Fluminense \\ 24210201 Niter\'{o}i \\ Brazil \\ ralph@mat.uff.br}

\maketitle

\begin{abstract} A \emph{gauge} $\gamma$ in a vector space $X$ is a distance function given by the Minkowski functional associated to a convex body $K$ containing the origin in its interior. Thus, the outcoming concept of \emph{gauge spaces} $(X, \gamma)$ extends that of finite dimensional real Banach spaces by simply neglecting the symmetry axiom (a viewpoint that Minkowski already had in mind). If the dimension of $X$ is even, then the fixation of a symplectic form yields an identification between $X$ and its dual space $X^*$. The image of the polar body $K^{\circ}\subseteq X^*$ under this identification yields a (skew-)dual gauge on $X$. In this paper, we study geometric properties of this so-called \emph{dual gauge}, such as its behavior under isometries and its relation to orthogonality. A version of the Mazur-Ulam theorem for gauges is also proved. As an application of the theory, we show that closed characteristics of the boundary of a (smooth) convex body are optimal cases of a certain isoperimetric inequality. 
\end{abstract}


\bigskip


\section{Introduction}

\emph{Gauges} are, roughly speaking, norms without symmetry axiom (sometimes also called \emph{asymmetric norms}). Vector spaces endowed with gauges have been extensively studied in the last few years (see, for example, the recent papers \cite{Bra-Mer}, \cite{Bra-Mer-Jahn-Mar}, \cite{Jahn1}, \cite{jahn}, \cite{Jahn2}, \cite{Mer-Jahn-Rich} and \cite{obst}). The aim of this work is to investigate duality for these spaces. We start with some basic notions. Let $X$ be a finite-dimensional vector space with origin $o$ (sometimes also denoted by $0_X$), and let $K \subseteq X$ be a \emph{convex body} (that is, a compact, convex set with non-empty interior) such that $o \in \mathrm{int}K$. The \emph{gauge} associated to $K$ is the functional $\gamma_K:X\rightarrow \mathbb{R}$ defined as
\begin{align*} \gamma_K(x) = \inf\{\lambda \geq 0:x\in \lambda K\},
\end{align*}
for $x \in X$ (see Figure \ref{gauge}). Of course, a gauge has the following properties:\\

\noindent\textbf{i.\,(positivity and nondegeneracy)} $\gamma_K(x) \geq 0$ for all $x \in X$, with equality if and only if $x = 0$, \\

\noindent\textbf{ii.\,(positive homogeneity)} $\gamma_K(\alpha x) = \alpha\gamma_K(x)$ for any $\alpha \geq 0$ and $x \in X$, \\

\noindent\textbf{iii.\,(triangle inequality)} $\gamma_K(x+y) \leq \gamma_K(x) + \gamma_K(y)$ for any $x,y \in X$. \\

\begin{figure}[h]
\centering
\includegraphics{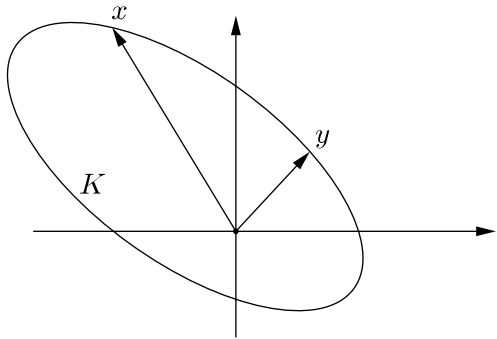}
\caption{$\gamma_K(x) = \gamma_K(y)= 1$.}
\label{gauge}
\end{figure}

Conversely, if $\gamma:X\rightarrow \mathbb{R}$ is a function satisfying \textbf{i}, \textbf{ii}, and \textbf{iii}, then the \emph{unit ball} $B_{\gamma} = \{x \in X : \gamma(x) \leq 1\}$ is a convex body with the origin in its interior and
\begin{align*} \gamma = \gamma_{B_{\gamma}},
\end{align*}  
from which $\gamma$ is a gauge. This shows that a gauge can be equivalently defined as the Minkowski functional of a convex body containing the origin in its interior (as it was originally done), or as a function satisfying \textbf{i}, \textbf{ii}, and \textbf{iii}. A vector space $(X,\gamma)$ endowed with a gauge will be called a \emph{gauge space}. 

The boundary $\partial B_{\gamma} = \{x \in X:\gamma(x) = 1\}$ of the unit ball of a gauge space $(X,\gamma)$ is called the \emph{unit sphere}. A gauge $\gamma_K$ defines a \emph{convex distance function} $d_K:X\times X\rightarrow[0,+\infty)$ on $X$ by
\begin{align*} d_K(x,y) = \gamma_K(y-x),
\end{align*}
for $x,y \in X$. Notice that $d_K$ is \textbf{not} necessarily symmetric. If this is the case, then the convex body $K$ is centered at the origin and the gauge is a usual \emph{norm}. And if this is not the case,  
then the unit ball still can be centrally symmetric, but its center has to differ from the origin (see again Figure \ref{gauge}, where the unit ball is centrally symmetric, but not centered at the orgin). On the other hand, it is clear that $d_K$ is translation invariant. It is easy to see that the open metric balls
\begin{align*}  \{x \in X: d_K(p,x) < r\},
\end{align*}
for $p \in X$ and $r>0$, induce on $X$ the same topology as any inner product or any norm induce. This happens because any Euclidean ball contains a homothetic copy of $K$, and any open metric ball as above contains a homothetic copy of the Euclidean ball. We also define the \emph{distance} between two subsets $A,B\subseteq (X,\gamma)$ as
\begin{align*} d_X(A,B) := \inf\{d_X(a,b):a \in A \ \mathrm{and} \ b\in B\},
\end{align*}
where we notice very carefully that this is not a symmetric concept. From standard convexity arguments, we get that the distance from a point to a line in a gauge space has a similar geometric interpretation as in the symmetric (norm) case. In what follows, recall that we denote by $B_p(r)$ the (closed) ball $p + rK$ with center $p \in X$ and radius $r > 0$, where $K$ is the unit ball of $\gamma$. 
\begin{prop} Let $\ell \subseteq (X,\gamma)$ be a line, and let $p \in X$ be a point such that $p \notin \ell$. Then
\begin{align*} d_X(p,\ell) = \inf\{r>0: B_p(r)\cap \ell \neq \emptyset\}.
\end{align*}
Moreover, if $r = d(p,\ell)$ and $q \in B_p(r)\cap\ell$, then $\ell$ supports $B_p(r)$ at $q$. 
\end{prop}
\begin{proof} If $B_p(r)\cap\ell = \emptyset$, then it is clear that $d_X(p,\ell) > r$. On the other hand, if $B_p(r) \cap \ell \neq \emptyset$, but $\ell$ does not support $B_p(r)$, then there exists a point $q \in \mathrm{int}B_p(r)\cap \ell$, from where $d_X(p,\ell) < r$. It follows that the equality $d_X(p,\ell) = r$ holds precisely when $\ell$ supports the ball $B_p(r)$. Our claims follow immediately from this argument, and Figure \ref{distpointline} illustrates the situation.

\end{proof}
\begin{figure}[h]
\centering
\includegraphics{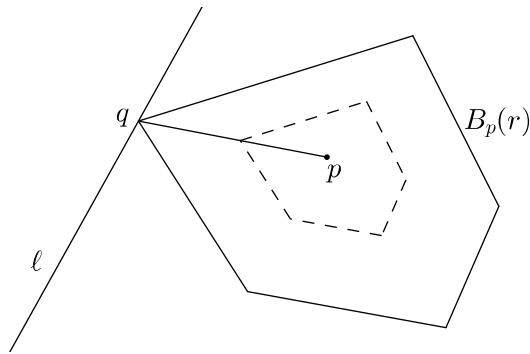}
\caption{The distance $d_X(p,\ell)$ is attained at $q \in \ell$.}
\label{distpointline}
\end{figure}

From now on, for simplicity of the notation, we will denote by $\mathcal{K}_{\mathrm{o}}(X)$ the space of convex bodies of $X$ which contain the origin as an interior point.

\section{The Mazur-Ulam theorem for gauges}

The Mazur-Ulam theorem states that any isometry between normed spaces which fixes the origin is linear (see \cite{thompson}). In this section, we extend this result to gauge spaces. Let $(X,\gamma_X)$ and $(Y,\gamma_Y)$ be gauge spaces. A map $T:(X,\gamma_X)\rightarrow (Y,\gamma_Y)$ is a \emph{gauge isometry} (or simply an \emph{isometry}) if
\begin{align*} d_X(x,z) = d_Y(Tx,Tz),
\end{align*}
for any $x,z \in X$, where $d_X$ and $d_Y$ are the distances induced by $\gamma_X$ and $\gamma_Y$, respectively. We also say that two gauges $\gamma_1$ and $\gamma_2$ in the same vector space $X$ are \emph{isometric} if there exists an isometry $T:(X,\gamma_1)\rightarrow(X,\gamma_2)$. From the non-degeneracy of gauges, we have that any gauge isometry is injective. If $T:(X,\gamma_X)\rightarrow (Y,\gamma_Y)$ is a surjective isometry, then we say that $(X,\gamma_X)$ and $(Y,\gamma_Y)$ are \emph{isometric gauge spaces}. 

Notice that an isometry is not necessarily gauge-preserving (although the converse is true). In particular, any translation is an isometry. 

\begin{teo}[Mazur-Ulam theorem for gauges] Let $(X,\gamma_X)$ and $(Y,\gamma_Y)$ be gauge spaces. If $T:(X,\gamma_X)\rightarrow(Y,\gamma_Y)$ is an isometry such that $T(o_X) = o_Y$, then $T$ is linear.
\end{teo}
\begin{proof} The idea is to construct norms on $X$ and $Y$ for which $T$ is also an isometry. Define
\begin{align*} ||x||_X = \gamma_X(x) + \gamma_X(-x)
\end{align*}
for each $x \in X$, and define $||\cdot||_Y$ analogously for vectors in $Y$. It is straightforward that $||\cdot||_X$ and $||\cdot||_Y$ are norms on $X$ and $Y$, respectively. Now let $T$ be an isometry between the gauge spaces $(X,\gamma_X)$ and $(Y,\gamma_Y)$. We have
\begin{align*} ||z - x||_X = \gamma_X(z-x) + \gamma_X(-z+x) = d_X(x,z) + d_X(z,x) = \\
d_Y(Tx,Tz) + d_Y(Tz,Tx) = \gamma_Y(Tz-Tx) + \gamma_Y(-Tz+Tx) =  ||Tz - Tx||_Y,
\end{align*}
for any $x,z \in X$. Therefore, $T$ is also an isometry between the normed spaces $(X,||\cdot||_X)$ and $(Y,||\cdot||_Y)$. Under the condition that $T(o_X) = o_Y$, the classical Mazur-Ulam theorem implies that $T$ is linear. This concludes the proof. 

\end{proof}

For more on the Mazur-Ulam theorem, including some generalizations, we refer the reader to the papers \cite{rassias} and \cite{wang}. It is clear that, among the maps which fix the origin, the isometries are precisely the gauge-preserving transformations.

Next we investigate what happens when the isometry does not fix the origin. We will show that, in this case, the isometry is linear up to composition with a translation. Recall that an \emph{affine map} (or \emph{affine transformation}) between vector spaces is a map which can be written as the composition of a translation with a linear map. We also say that two convex bodies $K_1\subseteq X$ and $K_2 \subseteq Y$ are \emph{affinely equivalent} if there exists an injective affine transformation $A:X\rightarrow Y$ such that $K_2 = A(K_1)$. The bodies are said to be \emph{linearly equivalent} if there exists a linear map $T:X\rightarrow Y$ such that $K_2 = T(K_1)$. 

\begin{coro} Any isometry between gauge spaces is an (injective) affine transformation. 
\end{coro} 
\begin{proof} Let $T:(X,\gamma_X)\rightarrow(Y,\gamma_Y)$ be a gauge isometry such that $T(o_X) = y_0 \in Y$. Let $T_0:X\rightarrow Y$ be defined as 
\begin{align*} T_0(x) = T(x) - y_0,
\end{align*}
for any $x \in X$. It is clear that $T_0$ is an isometry which fixes the origin, and hence $T_0$ is linear. Writing $T = T_0 + y_0$ gives the result.

\end{proof}

\begin{remark}\label{lineariso} In a certain way, every isometry can be regarded as linear. If $T:(X,\gamma_X)\rightarrow(Y,\gamma_Y)$ is an affine isometry written as $T = T_0 + y_0$, then the linear map $T_0$ is an isometry between $(X,\gamma_X)$ and $(Y,\gamma_Y)$. Indeed, for any $x,z \in X$ we have
\begin{align*} d_X(x,z) = \gamma_X(z-x) = \gamma_Y(Tz-Tx) = \gamma_Y(T_0z - T_0x) = d_Y(T_0x,T_0z).
\end{align*}
In particular, if two gauge spaces are isometric, then they are also \emph{linearly isometric}, meaning that there exists a linear isometry between them.
\end{remark}

\begin{coro}\label{isoaff} Two gauges $\gamma_1$ and $\gamma_2$ on a vector space $X$ are isometric if and only if their unit balls $K_1$ and $K_2$ are linearly equivalent convex bodies. 
\end{coro}
\begin{proof} Saying that $\gamma_1$ and $\gamma_2$ are isometric is precisely the same as stating that there exists a linear isometry $T:(X,\gamma_1)\rightarrow(X,\gamma_2)$. It is clear that
\begin{align*} \gamma_1(x) \leq 1 \Leftrightarrow \gamma_2(Tx) \leq 1,
\end{align*}
and hence $T(K_1) = K_2$. Since $T$ must be an injective affine transformation, we have that $K_1$ and $K_2$ are linearly equivalent. 

\end{proof}

\begin{remark} In the symmetric case, we can state that two norms are isometric if and only if their unit balls are affinely equivalent. This happens because any affine transformation mapping a centrally symmetric convex body onto another centrally symmetric convex body must be linear. In the asymmetric case, however, a ``small" translation of the unit ball leads to a non-isometric gauge, as we will see later. 
\end{remark}

\begin{lemma}\label{liniso} Let $(X,\gamma_X)$ and $(Y,\gamma_Y)$ be gauge spaces, and let $T:X\rightarrow Y$ be an isomorphism such that $T(K_X) = K_Y$, where $K_X$ and $K_Y$ are the unit balls of $(X,\gamma_X)$ and $(Y,\gamma_Y)$, respectively. Then $T$ is a gauge isometry.
\end{lemma}
\begin{proof} Let $x \in X$ be a non-zero vector. If $\gamma_X(x) = \alpha$, then $x/\alpha \in \partial K_X$, and since $T$ clearly takes the boundary of $K_1$ onto the boundary of $K_2$, we get that $T(x/\alpha) \in \partial K_2$. Hence
\begin{align*} \gamma_1(x) = \alpha\gamma_1\!\left(\frac{x}{\alpha}\right)\! = \alpha\gamma_2\!\left(\frac{Tx}{\alpha}\right)\! = \gamma_2(Tx).
\end{align*}
It follows that $T$ is gauge-preserving (the case where $x = o_X$ is trivial), and hence it is an isometry.

\end{proof}

\begin{prop}\label{noniso} Let $K \in \mathcal{K}_{\mathrm{o}}(X)$ induce a gauge $\gamma$ on  $X$. If $v \in X$ is a non-zero vector such that $K_2 := K+v \in \mathcal{K}_{\mathrm{o}}(X)$, then the gauge $\gamma_2$ given by $K_2$ is not isometric to $\gamma$.
\end{prop}
\begin{proof} Assume that there exists an isometry $T:(X,\gamma)\rightarrow(X,\gamma_2)$, and write $T= T_0 + y_0$ for some linear transformation $T_0$ and some vector $y_0 \in X$. Denote by $d_1$ and $d_2$ the distances given by $\gamma$ and $\gamma_2$, respectively. For any $x \in \partial K$, we have 
\begin{align*} d_2(T(o_X),T(x)) = d_2(y_0,T_0x+y_0) = \gamma_2(T_0x).
\end{align*}
On the other hand, we have $d_1(o_X,x) = \gamma(x) = 1$. It follows that $\gamma_2(T_0x) = 1$ for any $x \in \partial K$. Consequently, we have the inclusion
\begin{align*} T_0(\partial K) \subseteq \partial K + v,
\end{align*}
which is clearly a contradiction.

\end{proof}

\section{Polar gauges and dual gauges}\label{ident}

The \emph{dual space} $X^*$ of a finite-dimensional real vector space $X$ is the set
\begin{align*}X^* = \{f:X\rightarrow\mathbb{R}:f \ \mathrm{is \ linear}\}.
\end{align*}
It is clear that the dual space is a vector space with the same dimension as the original space. If $(X,\gamma)$ is a gauge space, then we can endow $X^*$ with the \emph{polar gauge} $\gamma^*:X^*\rightarrow\mathbb{R}$ defined as
\begin{align*} \gamma^*(f) = \max\{f(x):x \in K\},
\end{align*}
where $K\subseteq X$ is the unit ball of $\gamma$ (that is, $\gamma = \gamma_K$). Of course, it is not immediate that $\gamma^*$ is a gauge, and hence we prove this in the next proposition.

\begin{prop} The map $\gamma^*:X^*\rightarrow\mathbb{R}$ defined as above is a gauge. 
\end{prop}
\begin{proof} Positivity, nondegeneracy, and positive homogeneity are immediate. We have to verify the triangle inequality. Let $f,g \in X^*$. Of course, for any $x \in K$ we have
\begin{align*} f(x) + g(x) \leq \gamma^*(f) + \gamma^*(g).
\end{align*}
Hence, taking the maximum of the left hand side over $x \in K$ we obtain that
\begin{align*}\gamma^*(f+g) \leq \gamma^*(f) + \gamma^*(g),
\end{align*}
as we wanted.

\end{proof}

Let $K \in \mathcal{K}_{\mathrm{o}}(X)$. The \emph{polar body} $K^{\circ}$ of $K$ is the convex body of the dual space $X^*$ defined as
\begin{align*} K^{\circ} = \{f \in X^*:f(x) \leq 1, \ \forall \ x \in K\}.
\end{align*}
We refer the reader to \cite{schneider} for more on polar bodies. In particular, we will assume without a proof that $K^{\circ} \in \mathcal{K}_{\mathrm{o}}(X^*)$. We also have the following result which will be useful later.

\begin{lemma}\label{dualpolar} Let $K \in \mathcal{K}_{\mathrm{o}}(X)$. Then $x \in K$ if and only if $f(x) \leq 1$ for any $f \in K^{\circ}$. 
\end{lemma}
\begin{proof} From the definition, $x \in K$ implies that $f(x) \leq 1$ for any $f \in K^{\circ}$. Hence we only have to prove the converse. If $x \notin K$, then $\lambda x \in K$ for some $\lambda < 1$ (recall that $K$ contains the origin as an interior point). Let $H$ be a hyperplane which supports $K$ at $\lambda x$. Since $X = \mathrm{span}\{x\}\oplus H$, we get that there exists a unique linear functional $f \in X^*$ such that $f(\lambda x) = 1$ and $f(y) = 0$ for any $y \in H$. From convexity we have $f(z) \leq 1$ for any $z \in K$, and thus $f \in K^{\circ}$. However, since $f(\lambda x) = 1$, we get
\begin{align*} f(x) = \frac{1}{\lambda} > 1.
\end{align*}
This contradiction concludes the proof. 

\end{proof}

\begin{prop} Let $\gamma = \gamma_K$ be a gauge in $X$. Then the unit ball of the polar gauge $\gamma^*$ is the polar body $K^{\circ}$ of $K$. 
\end{prop}
\begin{proof} Denote by $B_{\gamma^*}$ be unit ball of the polar gauge $\gamma^*$. If $f \in B_{\gamma^*}$, then $\gamma^*(f) \leq 1$, which implies that $f(x) \leq 1$ for any $x \in K$. Hence $f \in K^{\circ}$. Conversely, if $f \in K^{\circ}$, then $f(x) \leq 1$ for any $x \in K$. Therefore, $\gamma^*(f) \leq 1$. It follows that $B_{\gamma^*} = K^{\circ}$. 

\end{proof}

As expected, we can also prove that the polar spaces associated to isometric gauge spaces are isometric. In what follows, we recall that the \emph{adjoint} of a linear map $T:X\rightarrow Y$ is the (linear) map $T^*:Y^*\rightarrow X^*$ defined as
\begin{align*} T^*\!(f)(x) = f(T(x)),
\end{align*}
for any $f \in Y^*$ and $x \in X$. 

\begin{prop}\label{polariso} If $(X,\gamma_X)$ and $(Y,\gamma_Y)$ are isometric gauge spaces, then their respective polar gauge spaces $(X^*,\gamma_X^*)$ and $(Y^*,\gamma_Y^*)$ are also isometric. Moreover, if $T:(X,\gamma_X)\rightarrow(Y,\gamma_Y)$ is a linear isometry, then $T^*:(Y,\gamma_Y^*)\rightarrow (X,\gamma_X^*)$ is a (linear) isometry.
\end{prop}
\begin{proof} By Remark \ref{lineariso}, if $(X,\gamma_X)$ and $(Y,\gamma_Y)$ are isometric, then they are linearly isometric. As usual, denote by $K_X$ and $K_Y$ the unit balls of the gauges $\gamma_X$ and $\gamma_Y$, respectively. Let $T:X\rightarrow Y$ be a linear isometry, and denote by $T^*$ its adjoint. Then, since $T(K_X) = K_Y$, we have
\begin{align*} \gamma_X^*(T^*\!f) = \max\{f\circ T(x):x \in K_X\} = \max\{f(y):y \in K_Y\} = \gamma_Y^*(f),
\end{align*}
for any $f \in Y^*$. Hence $T^*$ is gauge-preserving. 

\end{proof}

If $X$ is an even-dimensional vector space, then the fixation of a nondegenerate bilinear form $\omega:X\times X\rightarrow\mathbb{R}$ (a \emph{symplectic form}) yields an identification between $X$ and $X^*$
by contraction in the first coordinate. That is, for any $f \in X^*$ there exists a unique vector $x_f \in X$ such that
\begin{align*} f(\cdot) = \iota_{x_f}\omega(\cdot) = \omega(x_f,\cdot).
\end{align*}
This yields an isomorphism $\mathcal{I}:X^*\rightarrow X$. Let $\gamma_K$ be the gauge defined by a convex body $K\in\mathcal{K}_{\mathrm{o}}(X)$. The image $\mathcal{I}({K^{\circ}})$ of the polar body $K^{\circ}$ under the identification $\mathcal{I}$ of $X^*$ and $X$ given by $\omega$ (and defined above) will be denoted by $K^{\omega}$ and will be called the \emph{dual body} of $K$. Notice that since the polar body satisfies $K^{\circ} \in \mathcal{K}_{\mathrm{o}}(X^*)$, it follows that $K^{\omega} \in \mathcal{K}_{\mathrm{o}}(X)$. Hence the dual body induces a gauge by
\begin{align*} \gamma_{K^{\omega}}(x) = \inf\{\lambda \geq 0 : x \in \lambda K^{\omega}\},
\end{align*}
for $x \in X$. This is called the \emph{dual gauge} of $\gamma_K$. The next proposition asserts that the duality of polarity holds up to the sign under the identification given by the symplectic form. The dual gauge is an extension of the concept of \emph{anti-norm} for usual norms (see \cite{Mar-Swa}).

\begin{prop}\label{bidual} We have
\begin{align*}\gamma_{K^{\omega}}(x) = \max\{\omega(x,y):y \in K\}
\end{align*}
for any $x \in X$. Moreover, the dual gauge of $\gamma_{K^{\omega}}$ is $\gamma_{-K}$.  
\end{prop}
\begin{proof} First, notice that the map
\begin{align*} x \mapsto \max\{\omega(x,y):y \in K\}
\end{align*}
is a gauge in $X$. Hence it suffices to prove that its unit ball $B^{\omega}$ is the dual body $K^{\omega}$ of $K$. Let $x \in K^{\omega}$, and assume that $x = \mathcal{I}(f)$ for some $f \in K^{\circ}$. Then $f(y) \leq 1$ for any $y \in K$, and thus
\begin{align*} \omega(x,y) = \omega(\mathcal{I}(f),y) = f(y) \leq 1,
\end{align*}
for any $y \in K$. It follows that 
\begin{align*} \max\{\omega(x,y):y \in K\} \leq 1,
\end{align*}
and then $x \in B^{\omega}$. This shows that $K^{\omega} \subseteq B^{\omega}$. Assume now that $x \in B^{\omega}$, and let $f \in X^*$ be such that $x = \mathcal{I}(f)$. For any $y \in K$ we have
\begin{align*} f(y) = \omega(\mathcal{I}(f),y) = \omega(x,y) \leq 1,
\end{align*}
where the last inequality is justified since $x \in B^{\omega}$. It follows that $f \in K^{\circ}$, and hence $x = \mathcal{I}(f) \in K^{\omega}$. Thus we also have the reverse inclusion $B^{\omega} \subseteq K^{\omega}$. 

It remains to prove that the dual gauge of $\gamma_{K^{\omega}}$ is $\gamma_{-K}$. By definition, the dual gauge of $\gamma_{K^{\omega}}$ is the gauge whose unit ball is the convex body $(K^{\omega})^{\omega} = \mathcal{I}((K^{\omega})^{\circ})$, that is, the identification under $\mathcal{I}$ of the polar body of $K^{\omega}$. Thus we have to show that $(K^{\omega})^{\omega} = -K$. Let $x \in (K^{\omega})^{\omega}$. If $f \in K^{\circ}$, then $\mathcal{I}(f) \in K^{\omega}$, and hence
\begin{align*} f(-x) = \omega(x,\mathcal{I}(f)) \leq \max\{\omega(x,y):y \in K^{\omega}\} = \gamma_{(K^{\omega})^{\omega}}(x) \leq 1.
\end{align*}
Since this holds for any $f \in K^{\circ}$, we get from Lemma \ref{dualpolar} that $-x \in K$. This gives the inclusion $(K^{\omega})^{\omega} \subseteq -K$. To prove the reverse inclusion, let $x \in -K$. If $y \in K^{\omega}$, then $y = \mathcal{I}(f)$ for some $f \in K^{\circ}$, and thus
\begin{align*} \omega(x,y) = \omega(y,-x) = \omega(\mathcal{I}(f),-x) = f(-x) \leq 1,
\end{align*}
where the last inequality is justified since $-x \in K$ and $f \in K^{\circ}$. It follows that
\begin{align*} \gamma_{(K^{\omega})^{\omega}}(x) = \max\{\omega(x,y):y \in K^{\omega}\} \leq 1,
\end{align*}
which yields $x \in (K^{\omega})^{\omega}$. Therefore, we get that $-K = (K^{\omega})^{\omega}$, and hence the dual gauge of $\gamma_{K^{\omega}}$ is $\gamma_{-K}$. 

\end{proof}

\begin{coro} If $K^{\omega} = \alpha K$ for some $\alpha \neq 0$ and some symplectic form $\omega$ on $X$, then the gauge $\gamma_K$ is a norm. Moreover, if $X$ is two-dimensional, then $\gamma_K$ is a Radon norm. 
\end{coro}
\begin{proof} If $K^{\omega} = \alpha K$, then, up to rescaling $\omega$, we may assume that $K^{\omega} = K$ (this is given in detail in Proposition \ref{rescale} below). Thus we have
\begin{align*} -K = (K^{\omega})^{\omega} = K^{\omega} = K,
\end{align*}
which gives that $K$ is centrally symmetric. In the two-dimensional case, $K^{\omega}$ is precisely the unit anti-ball (see \cite{Mar-Swa}), from which it follows that $\gamma_K$ is a Radon norm. 

\end{proof}

\begin{remark} If $\gamma$ is a norm in a plane, then the dual norm of $\gamma$ is the associated \emph{anti-norm} (again, we refer the reader to \cite{Mar-Swa}). Of course, in this case the symmetry guarantees that $K = -K$ (where $K$ is the unit ball), and hence the anti-norm of the anti-norm is the original norm. 
\end{remark}

A natural question that arises is whether the hypothesis of the previous corollary can be relaxed: must $K$ be centrally symmetric provided the dual gauge $\gamma_{\omega}$ is isometric to $\gamma_K$? The answer is \textbf{no}, as the example below shows. 

\begin{example} Consider the Euclidean space $\mathbb{R}^2$ endowed with the gauge $\gamma$ given by the equilateral triangle $K$ whose vertices are the points $(0,2)$, $(\sqrt{3},-1)$, and $(-\sqrt{3},-1)$; notice that its barycenter is the origin. To make our argument easier, we identify $\mathbb{R}^2$ with its dual space $\mathbb{R}_2$ by using the standard inner product $\langle\cdot,\cdot\rangle$. In this case, the identification of the polar body $K^{\circ}$ in $\mathbb{R}^2$ becomes
\begin{align*} K^{\circ} = \{y \in \mathbb{R}^2:\langle y,x\rangle \leq 1 \ \mathrm{for \ all} \ x \in K\},
\end{align*}
where we are abusing a little of the notation, since originally we have that the polar body is a convex body of the dual space. With this identification, it is easy to see that the polar body of the equilateral triangle $K$ is the equilateral triangle whose vertices are the midpoints of the edges of $K$ (whose barycenter is also the origin). Now let $\omega = \mathrm{det}$ be the usual determinant of $\mathbb{R}^2$. Since
\begin{align*} \langle e_1,\cdot\rangle = \mathrm{det}(-e_2,\cdot) \ \mathrm{and}\\
\langle e_2,\cdot \rangle = \mathrm{det}(e_1,\cdot),
\end{align*}
it follows that the identification between $\mathbb{R}^2$ and its dual space given by the determinant is simply the clockwise $\pi/2$-rotation of the identification given by the inner product. Thus we obviously have
\begin{align*} K = 2\mathrm{id}\circ \mathrm{rot}_{3\pi/2}(K^{\omega}),
\end{align*}
where $\mathrm{id}$ is the identity transformation of $\mathbb{R}^2$ and $\mathrm{rot}_{3\pi/2}$ is the counter-clockwise $3\pi/2$-rotation of the plane. Notice that $\mathrm{rot}_{3\pi/2}$ is the composition of the counter-clockwise $\pi/2$-rotation with the reflexion through the $x$-axis. This is illustrated in Figure \ref{equilateral}.

\end{example}

\begin{figure}[h]
\centering
\includegraphics{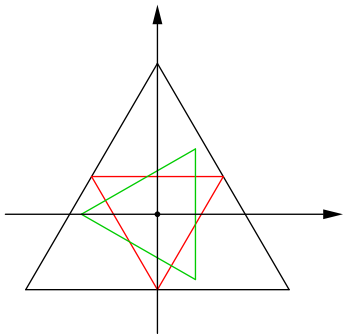}
\caption{The bodies $K$ (black), $K^{\circ}$ (red) and $K^{\omega}$ (green).}
\label{equilateral}
\end{figure}

Another natural question is whether an isometry of gauge spaces is also an isometry between their respective dual gauge spaces. We have that the dual gauges of isometric gauge spaces are isometric, but not necessarily by the \textbf{same} isometry. Recall that a \emph{symplectic linear map} (or \emph{linear symplectomorphism}) $T$ between symplectic vector spaces $(X,\omega_X)$ and $(Y,\omega_Y)$ is a linear map such that $\omega_Y(Tx,Ty) = \omega_X(x,y)$ for any $x,y \in X$. 

\begin{teo}\label{gaugeiso} Let $T:(X,\gamma_X,\omega_X)\rightarrow (Y,\gamma_Y,\omega_Y)$ be a linear isometry of gauge spaces. If $T$ is a symplectic linear map, then $T$ is an isometry between the respective dual gauges $\gamma_{\omega_X}$ and $\gamma_{\omega_Y}$. If $T$ is not symplectic, then the dual gauges are still isometric, but not necessarily by $T$. 
\end{teo}
\begin{proof} If an isometry $T:(X,\gamma_X)\rightarrow(Y,\gamma_Y)$ is also a symplectic linear map with respect to the fixed symplectic forms, then for any $x \in X$ we get
\begin{align*} \gamma_{\omega_Y}(Tx) = \max\{\omega_Y(Tx,y):y \in K_Y\} = \max\{\omega_Y(Tx,Tz):z \in K_X\} = \\ = \max\{\omega_X(x,z):z \in K_X\} = \gamma_{\omega_X}(x),
\end{align*}
where $K_Y$ and $K_X$ denote the unit balls of $\gamma_X$ and $\gamma_Y$, respectively. 

For the other claim, let $\mathcal{I}_X:X^*\rightarrow X$ and $\mathcal{I}_Y:Y^*\rightarrow Y$ be the usual identifications given by $\omega_X$ and $\omega_Y$, respectively. Since $T$ is an isometry, we have that its adjoint $T^*:Y^*\rightarrow X^*$ is an isometry between the polar gauges. Hence the map
\begin{align*} T^{\omega} := \mathcal{I}_X\circ T^*\circ\mathcal{I}_Y^{-1}:(Y,\gamma_{\omega_Y})\rightarrow (X,\gamma_{\omega_X})
\end{align*}
is an isometry of the dual gauges. Indeed, this is a composition of isometries. If $\gamma^*_X$ and $\gamma^*_Y$ denote the polar gauges in $X^*$ and $Y^*$, respectively, then
\begin{align*} \gamma_{\omega_X}(T^{\omega}y) = \gamma_X^*(T^*\circ\mathcal{I}_Y^{-1}y) = \gamma_Y^*(\mathcal{I}_Y^{-1}y) = \gamma_{\omega_Y}(y),
\end{align*}
for any $y \in Y$.

\end{proof}

\section{An orthogonality concept}

Based on Birkhoff orthogonality for normed spaces, we will define and briefly study an orthogonality relation between vectors in a gauge space $(X,\gamma)$. We say that a vector $x \in X$ is \emph{orthogonal} to a vector $y \in X$ (denoted by $x \dashv y$) whenever
\begin{align*} \gamma(x) \leq \gamma(x + ty),
\end{align*}
for any $t \in \mathbb{R}$. The reader may notice that this is a direct analogue of Birkhoff orthogonality for normed spaces (see the survey \cite{Alo-Mar-Wu} for orthogonality types in normed spaces). In the next proposition we state and prove some early properties of this orthogonality concept. For more on orthogonality concepts for gauges we refer the reader to \cite{jahn}. 

\begin{figure}[h]
\centering
\includegraphics{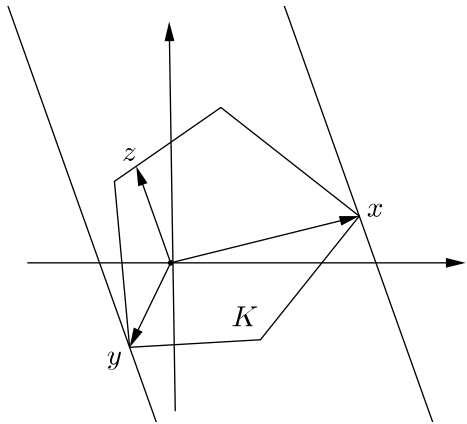}
\caption{$x \dashv z$ and $y \dashv z$.}
\label{ortho}
\end{figure}

\begin{prop} The orthogonality for gauge spaces has the following properties:\\

\noindent\emph{\textbf{i.\,(geometric interpretation)}} Let $x,y \in X$ be non-zero vectors, and assume that $\gamma(x) = r$. Then $x \dashv y$ if and only if the line $\{x+ty:t \in \mathbb{R}\}$ supports the planar disk $B_{\gamma}(r)\cap\mathrm{span}\{x,y\}$, where $B_{\gamma}(r) := \{z \in X : \gamma(z) \leq r\}$, at $x$ (see Figure \ref{ortho}). \\

\noindent\emph{\textbf{ii.\,(right-homogeneity)}} If $x\dashv y$, then $x \dashv \lambda y$ for any $\lambda \in \mathbb{R}$. \\

\noindent\emph{\textbf{iii.\,(positive left-homogeneity)}} If $x \dashv y$, then $\alpha x\dashv y$ for any $\alpha \geq 0$. 
\end{prop}
\begin{proof} To prove \textbf{i}, we first assume that $x\dashv y$ for non-zero $x,y \in X$ with $\gamma(x) = r$, and write $B = B_{\gamma}(r)\cap\mathrm{span}\{x,y\}$. We have to prove that the line $\{x+ty:t\in\mathbb{R}\}$ has no points in the (relative) interior of the planar body $B$. This comes straightforwardly from the inequality
\begin{align*} \gamma(x+ty) \geq \gamma(x) = r,
\end{align*}
which holds for any $t \in \mathbb{R}$ (due to the orthogonality relation). Since $\gamma(z) < r$ for any $z \in \mathrm{int}B$, we have the desired. Conversely, if the line $\{x+ty:t\in\mathbb{R}\}$ supports $B$ at $x$, then $\gamma(x+ty) \geq r$ for any $t \in \mathbb{R}$, implying $x \dashv y$.

For \textbf{ii}, just notice that if $\gamma(x) \leq \gamma(x+ty)$ for any $t \in \mathbb{R}$, then clearly we have that $\gamma(x) \leq \gamma(x+t\lambda y)$ for any $t \in \mathbb{R}$. The last property comes from the positive homogeneity of $\gamma$. If $x \dashv y$, then for any $\alpha \geq 0$ we have
\begin{align*} \gamma(\alpha x) = \alpha\gamma(x) \leq \alpha\gamma(x+ty) = \gamma(\alpha x + t\alpha y),
\end{align*}
for any $t \in \mathbb{R}$. If $\alpha > 0$, then the inequality above gives $\gamma(\alpha x) \leq \gamma(\alpha x + ty)$ for any $t \in \mathbb{R}$, and the case $\alpha = 0$ is trivial. 

\end{proof}

As a consequence of \textbf{i}, we have that if $H$ is a hyperplane which supports the ball $B_{\gamma}(r)$ at $x$ (with $r > 0$), then $x \dashv y$ for any $y \in H$. Conversely, if $x\dashv y$, then $y$ is a vector of some hyperplane which supports $B_{\gamma}(r)$ at $x$, where $r = \gamma(x)$. In this case we also write $x \dashv H$, and the orthogonality is extended to a relation between vectors and hyperplanes. 

\begin{coro} Let $H \subseteq X$ be a hyperplane. Then there exist at least two distinct vectors $x,y \in B_{\gamma}$ such that $x \dashv H$ and $y \dashv H$. Moreover, if for any hyperplane $H\subseteq X$ these vectors are equal up to the sign, then the gauge is a norm. 
\end{coro}
\begin{proof} The first claim comes directly from the fact that any pair of parallel supporting  
hyperplanes supports a given convex body at two boundary points, at least.

For the other claim, notice that if each pair of parallel supporting hyperplanes supports $B_{\gamma}$ at points representable by $x$ and $-x$, then $B_{\gamma}$ is centrally symmetric. Hence $\gamma$ is a norm.

\end{proof}

A gauge $\gamma_K$ is \emph{smooth} if its unit ball $K$ is a \emph{smooth convex body}, meaning that $K$ is supported by only one hyperplane at each point of its boundary $\partial K$. If $\gamma$ is a smooth gauge, then for each $x \in K$ there exists a unique hyperplane $H$ such that $x \dashv H$. We also say that a gauge is \emph{strictly convex} if the boundary of its unit ball contains no line segment. This means that for each hyperplane $H \subseteq X$ there are precisely two vectors $x,y \in K$ such that $x\dashv H$ and $y \dashv H$.

\begin{prop} A gauge is strictly convex if and only if the triangle inequality is strict for linearly independent vectors.
\end{prop}
\begin{proof} First, let $x,y \in X$ be linearly independent vectors and assume that $\gamma(x+y) = \gamma(x) + \gamma(y)$. This gives
\begin{align*} \frac{\gamma(x)}{\gamma(x+y)} + \frac{\gamma(y)}{\gamma(x+y)} = 1,
\end{align*}
and hence
\begin{align*}\frac{x+y}{\gamma(x+y)} = \frac{\gamma(x)}{\gamma(x+y)}\frac{x}{\gamma(x)} + \frac{\gamma(y)}{\gamma(x+y)}\frac{y}{\gamma(y)} \in \mathrm{seg}\!\left[\!\frac{x}{\gamma(x)},\frac{y}{\gamma(y)}\!\right]\!.
\end{align*}
It follows that the segment joining $x/\gamma(x)$ and $y/\gamma(y)$ contains a relatively interior point in the boundary of $B_{\gamma}$. From convexity we get that the entire segment is contained in $\partial B_{\gamma}$. 

For the converse, notice that if $\mathrm{seg}[x,y]$ is a non-degenerate segment contained in $\partial B_{\gamma}$, then we clearly have $\gamma(x+y) = 2 = \gamma(x) + \gamma(y)$. 

\end{proof}

Recall that a point $x \in \partial K$ is said to be \emph{non-smooth} (or \emph{singular}) if there is more than one hyperplane which supports $K$ at $x$. The \emph{face} of a point $x \in K$ is the set containing $x$ and all the points $y \in K$ for which the line determined by $x$ and $y$ has an open segment containing $x$. Of course, if $x \in \mathrm{int}K$, then $F_x = K$. If $x \in \partial K$, then the face $F_x$ is clearly contained in $\partial K$, and it is called a \emph{proper face}. A proper face $F_x$ is \emph{non-degenerate} if $F_x \neq \{x\}$. For more on faces of convex bodies we refer the reader to \cite[Chapter 2]{schneider} and \cite[Chapter 1]{Bol-Mar-Sol}.

\begin{prop} Every singular point $x \in \partial K$ corresponds to a non-degenerate face in $\partial K^{\omega}$, and vice versa. In particular, $K^{\omega}$ is strictly convex if and only if $K$ is smooth, and the contrary also holds. 
\end{prop}
\begin{proof} Assume that $K$ is supported at $x \in \partial K$ by distinct hyperplanes $H_1$ and $H_2$, and let $f_1,f_2\in X^*$ be linear functionals such that $f_1(x) = f_2(x) = 1$, $\mathrm{ker}f_1 = H_1$ and $\mathrm{ker}f_2 = H_2$. Hence it is clear that $f_1,f_2 \in \partial K^{\circ}$, since $\gamma^*(f_1) = \gamma^*(f_2) = 1$. If $\lambda \in (0,1)$, then we claim that $\lambda f_1 + (1-\lambda)f_2 \in \partial K^{\circ}$. Indeed,
\begin{align*} \gamma^*(\lambda f_1 + (1-\lambda)f_2) = \lambda\gamma^*(f_1) + (1-\lambda)\gamma^*(f_2) = 1.
\end{align*}
It follows that $f_1$ and $f_2$ lie in some non-degenerate face of $\partial K^{\circ}$ (namely, in the face of $(f_1+f_2)/2$). Since a face of a convex body is preserved by a linear map, it follows that $\partial K^{\omega}$ contains a non-degenerate face. 

On the other hand, the fact that a face of $\partial K$ corresponds to a non-smooth point of $\partial K^{\omega}$ comes straightforwardly from the (skew-)duality $(K^{\omega})^{\omega} = -K$.

\end{proof}

In the next two sections we investigate in more detail how support and orthogonality behave under duality. 

\section{The two-dimensional case}\label{planarorth}

A two-dimensional vector space endowed with a gauge is called a \emph{gauge plane}, and the corresponding unit ball and unit sphere are called \emph{unit disk} and \emph{unit circle}, respectively.  In the case of gauge planes, we can use orthogonality to characterize ``where the dual gauge is attained". This is what we do next.

\begin{prop}\label{ortmax} Let $(X,\gamma)$ be a gauge plane with a fixed symplectic form $\omega$. Let $\gamma_{\omega}$ denote the dual gauge of $\gamma$. For any non-zero vector $x \in X$ we have that
\begin{align*} \gamma_{\omega}(x) = \omega(x,y_0)
\end{align*}
for $y_0 \in \partial B_{\gamma}$ if and only if $y_0 \dashv x$ and $\omega(x,y_0) > 0$. 
\end{prop}
\begin{proof} Let $x$ be a non-zero vector, and assume that $y_0 \in B_{\gamma}$ is a vector such that $y_0 \dashv x$ and $\omega(x,y_0) >0$. This means that the line in the direction of $x$ supports the unit disk $B_{\gamma}$ at $y_0$. The (open) arc of $\partial B_{\gamma}$ which contains $y_0$ is the set
\begin{align*} \partial_x^+B_{\gamma}:=\!\left\{\!\frac{y_0+tx}{\gamma(y_0+tx)}:t \in \mathbb{R}\!\right\}\!,
\end{align*}
and it is easy to see that for a given vector $y \in \partial B_{\gamma}$ we have that $\omega(x,y) > 0$ if and only if $y \in \partial_x^+B_{\gamma}$. Thus,
\begin{align*} \gamma_{\omega}(x) = \max\{\omega(x,y):y \in \partial^+_xB_{\gamma}\} = \max\!\left\{\!\frac{\omega(x,y_0)}{\gamma(y_0+tx)}:t \in \mathbb{R}\!\right\}\!,
\end{align*}
and since $\gamma(y_0 + tx) \geq \gamma(y_0) = 1$ for any $t \in \mathbb{R}$, we get that $\gamma_{\omega}(x) = \omega(x,y_0)$. Conversely, assume that $y_0 \in B_{\gamma}$ is such that
\begin{align*} \gamma_{\omega}(x) = \omega(x,y_0).
\end{align*}
It is obvious that $\omega(x,y_0) > 0$, and hence we have to prove that $y_0 \dashv x$. For that sake notice the following: if the line $\{y_0+tx:t\in\mathbb{R}\}$ does not support $B_{\gamma}$ at $y_0$, then there exists a number $t_0 \in \mathbb{R}$ such that $y_0+t_0x \in \mathrm{int}B_{\gamma}$. In this case, we get that $\gamma(y_0+t_0x) < 1$, and thus
\begin{align*} \omega\!\left(\!x,\frac{y_0+t_0x}{\gamma(y_0+t_0x)}\!\right)\! = \frac{\omega(x,y_0)}{\gamma(y_0+t_0x)} > \omega(x,y_0),
\end{align*}
contradicting the fact that $\omega(x,y_0) = \max\{\omega(x,y):y \in B_{\gamma}\}$. It follows that the line $\{y_0+tx:t\in\mathbb{R}\}$ indeed supports $B_{\gamma}$ at $y_0$, and hence $y_0 \dashv x$. Figure \ref{dualattain} illustrates the situation.

\end{proof}

\begin{figure}
\centering
\includegraphics{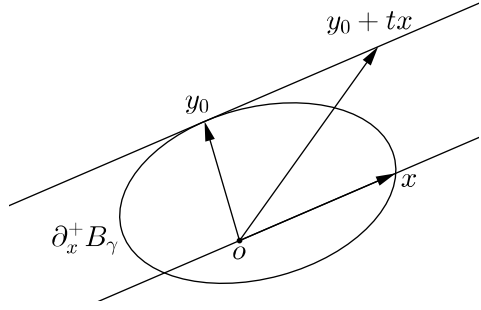}
\caption{$\gamma_{\omega}(x) = \omega(x,y_0)$.}
\label{dualattain}
\end{figure}

The next natural question is how our orthogonality concept behaves when we replace a gauge by its dual gauge (recall that the anti-norms reverse Birkhoff orthogonality). As a first step in this direction we notice that, as a consequence of Proposition \ref{ortmax}, we have an inequality involving the fixed symplectic form, the original gauge, and the induced dual gauge. This is explained next. The reader may notice that an analogue holds for norms and anti-norms.

\begin{prop}\label{ineqdual} For any $x,y \in X$ we have
\begin{align*} \omega(x,y) \leq \gamma_{\omega}(x)\gamma(y).
\end{align*}
If $(X,\gamma)$ is a gauge plane, then equality holds if and only if $y \dashv x$ and $\omega(x,y)  \geq 0$.
\end{prop}
\begin{proof} We can assume that $x,y \in X$ are non-zero vectors, since otherwise the result is trivial. We have that
\begin{align*} \gamma_{\omega}(x) = \max\{\omega(x,y):y \in K\} \geq \omega(x,y),
\end{align*}
for any $y \in B_{\gamma}$. This gives
\begin{align*} \gamma_{\omega}(x) \geq \omega\!\left(\!x,\frac{y}{\gamma(y)}\!\right)\!,
\end{align*}
for any non-zero vector $y \in X$. This inequality gives immediately $\omega(x,y) \leq \gamma_{\omega}(x)\gamma(y)$. 

For the equality case we can also assume that $x,y\in X$ are non-zero vectors. Indeed, we notice that
\begin{align*} y \dashv x \ \mathrm{and} \ \omega(x,y) = 0 \Leftrightarrow x = 0_X \ \mathrm{or} \ y = 0_X,
\end{align*}
where $0_X$ denotes the zero vector of $X$. If $x,y$ are non-zero vectors of a gauge plane $(X,\gamma)$, and $\omega(x,y) = \gamma_{\omega}(x)\gamma(y)$, then
\begin{align*} \gamma_{\omega}(x) = \omega\!\left(\!x,\frac{y}{\gamma(y)}\!\right)\!,
\end{align*}
and Proposition \ref{ortmax} gives that 
\begin{align*} \frac{y}{\gamma(y)} \dashv x \ \mathrm{and} \ \omega\!\left(\!x,\frac{y}{\gamma(y)}\!\right)\!> 0,
\end{align*}
which yields $y \dashv x$ and $\omega(x,y) > 0$ (recall that orthogonality is positively left-homogeneous). Conversely, if $x,y \in X$ are non-zero vectors such that $y \dashv x$ and $\omega(x,y) > 0$, then $z = y/\gamma(y)$ is such that $z \in B_{\gamma}$, $z \dashv x$ and $\omega(x,z) > 0$. Hence
\begin{align*} \gamma_{\omega}(x) = \omega(x,z) = \omega\!\left(\!x,\frac{y}{\gamma(y)}\!\right)\!,
\end{align*}
and then we get $\omega(x,y) = \gamma_{\omega}(x)\gamma(y)$. 

\end{proof}

In the case of non-symmetric gauge planes, the duality does not reverse the or\-tho\-go\-na\-li\-ty relation. This happens because the bi-dual gauge of $\gamma_K$ is $\gamma_{-K}$, as Proposition \ref{bidual} shows. The dual gauge reverses orthogonality \textbf{and} orientation, as we will see next. Before presenting this, we need a technical lemma.
\begin{lemma}\label{oppgauge} Let $K \in \mathcal{K}_{o}(X)$. Then $\gamma_K(-x) = \gamma_{-K}(x)$ for any $x \in X$. Moreover, if $\dashv_{K}$ denotes the orthogonality relation given by $\gamma_K$, then we have that
\begin{align*} x \dashv_K y \Leftrightarrow -x \dashv_{-K} y,
\end{align*}
for any vectors $x,y \in X$. 
\end{lemma}
\begin{proof} Assume that $z \in -K$. Then $\gamma_{-K}(z) = 1$. Since $-z \in K$, we have $\gamma_{K}(-z) = 1$. If $x \in X$ is a non-zero vector, then $x = \lambda z$ for some vector $z \in -K$. Hence
\begin{align*} \gamma_{-K}(x) = \lambda\gamma_{-K}(z) = \lambda\gamma_K(-z) = \gamma_K(-\lambda z) = \gamma_K(-x).
\end{align*}
The case where $x = 0_X$ is trivial. For the other claim, notice that $x \dashv_K y$ means that $\gamma_K(x) \leq \gamma_{K}(x+ty)$ for any $t \in \mathbb{R}$, and this is true if and only if 
\begin{align*} \gamma_{-K}(-x) \leq \gamma_{-K}(-x+ty)
\end{align*}
for any $t \in \mathbb{R}$, which is precisely the definition of $-x \dashv_{-K}y$. 

\end{proof}

\begin{teo} Let $(X,\gamma)$ be a gauge plane. As usual, denote by $K$ the unit ball of $\gamma$, by $\gamma_{\omega}$ the dual gauge of $\gamma$, and by $\dashv_{\omega}$ the orthogonality relation given by $\gamma_{\omega}$. For $x,y \in X$ with $\omega(y,x) > 0$, we have that $x \dashv_{\omega} y$ implies $-y \dashv x$. 
\end{teo}
\begin{proof} We may assume that $x,y \in X$ are non-zero vectors, and that $\omega(y,x) > 0$, since otherwise the result is trivial. Under these hypotheses, if $x \dashv_{\omega} y$, then Proposition \ref{ineqdual} (applied for $\gamma_{\omega}$ and its dual gauge) gives that
\begin{align*} \omega(y,x) = (\gamma_{\omega})_{\omega}(y)\cdot\gamma_{\omega}(x),
\end{align*}
where $(\gamma_{\omega})_{\omega}$ is, in particular, the bi-dual gauge of $\gamma$. From Proposition \ref{bidual}, we have that $(\gamma_{\omega})_{\omega} = \gamma_{-K}$, and hence the equality above yields
\begin{align*} \gamma_{\omega}(x) = \omega\!\left(\!x,\frac{-y}{\gamma_{-K}(y)}\!\right)\!.
\end{align*}
From Lemma \ref{oppgauge} we get
\begin{align*} \frac{-y}{\gamma_{-K}(y)} \in K.
\end{align*}
Hence Proposition \ref{ortmax} gives that 
\begin{align*} \frac{-y}{\gamma_{-K}(y)} \dashv x,
\end{align*}
and from the positive left-homogeneity it follows that $-y \dashv x$. 

\end{proof}

\begin{coro} Under the same hypotheses and notation like that of the last theorem, we get that if $x \dashv y$, then $y \dashv_{\omega} x$. 
\end{coro}
\begin{proof} If $x \dashv y$, then Lemma \ref{oppgauge} gives that $-x \dashv_{-K} y$, and hence $-x\dashv_{-K}-y$. Since $\dashv_{-K}$ is the bi-dual orthogonality relation, and $\omega(-y,-x) = \omega(y,x) > 0$ (recall that we are under the same hypotheses like the previous proposition), we get that $-x \dashv_{-K}-y$ implies $y \dashv_{\omega} -x$. Due to the right-homogeneity we get $y \dashv_{\omega} x$. 

\end{proof}

\begin{remark} Notice that the theorem could be analogously obtained from the corollary if we would prove the latter first. However, the assumption that $x \dashv_{\omega} y$ would create the easier way, since then we could use the ``trick" of applying Proposition \ref{ineqdual} to the dual gauge $\gamma_{\omega}$. 
\end{remark}

Our last task in this section is to investigate what happens when we change the symplectic form. Since we are working in dimension $2$, we know that two symplectic forms are equal up to scalar multiplication. In the symmetric case, this yields that two anti-norms obtained from different symplectic forms are essentially the same, since their unit disks will the homothetic (see \cite{Mar-Swa}). For gauge planes the situation is slightly more complicated, since a different sign gives a different geometry.

\begin{prop}\label{rescale} Let $\omega_1$ and $\omega_2$ be symplectic forms in a gauge plane $(X,\gamma)$, and assume that $\omega_2 = \alpha\omega_1$ for some number $\alpha \neq 0$. Then $K^{\omega_1} = \alpha K^{\omega_2}$. Consequently, if $\alpha > 0$, then $\gamma_{\omega_2} = \alpha\gamma_{\omega_1}$, and if $\alpha < 0$, then $\gamma_{\omega_2} = -\alpha\gamma_{-\omega_1}$, where $\gamma_{-\omega_1}$ is the gauge whose unit disk is the body $-K^{\omega_1}$. 
\end{prop}
\begin{proof} Denote by $\mathcal{I}_1:X^*\rightarrow X$ and $\mathcal{I}_2:X^*\rightarrow X$ the identifications given by $\omega_1$ and $\omega_2$, respectively. We have that $y \in K^{\omega_1}$ if and only if $y = \mathcal{I}_1(f)$ for some $f \in K^{\circ}$. In this case, 
\begin{align*} \omega_1(\mathcal{I}_1(f),x) = f(x) = \omega_2(\mathcal{I}_2(f),x) = \alpha\omega_1(\mathcal{I}_2(f),x) = \omega_1(\alpha\mathcal{I}_2(f),x),
\end{align*}
for any $x \in X$. Hence $y = \mathcal{I}_1(f) = \alpha\mathcal{I}_2(f)$. Since $\mathcal{I}_2(f) \in K^{\omega_2}$, we get that $y \in \alpha K^{\omega_2}$. The proof of the converse inclusion is similar, and the rest is straightforward. 

\end{proof}

\begin{coro} Let $\omega_1$ and $\omega_2$ be two symplectic forms on a gauge plane $(X,\gamma)$. If these forms have the same sign, then the associated orthogonality relations $\dashv_{\omega_1}$ and $\dashv_{\omega_2}$ coincide. If $\omega_1$ and $\omega_2$ have opposite signs, then $x \dashv_{\omega_1} y$ is equivalent to $-x \dashv_{\omega_2} y$. 
\end{coro}

\section{The higher (even-)dimensional case}

In the previous section we investigated how orthogonality is related to duality for gauge planes. Now we want to extend these investigations to higher dimensional vector spaces. We have to require that the considered gauge space is even-dimensional such that we can fix a symplectic form on it. For the case of even-dimensional normed spaces, we refer the reader to \cite{Hor-Spi-Lan}.

When working on a gauge plane $(X,\gamma)$, we know that two symplectic forms are equal up to scalar multiplication, and hence it is not difficult to understand what happens when one changes the symplectic form. In higher (even-)dimensional vector spaces, we have ``more freedom" to choose a symplectic form, and then we expect duality to be more complicated. The first thing we can ask is whether two dual bodies of $K$ (corresponding to different symplectic forms) yield the same gauge geometry. We will prove that this is indeed true. 

\begin{teo} Let $\omega_1$ and $\omega_2$ be symplectic forms on an even-dimensional vector space $X$, and let $K\in\mathcal{K}_{\mathrm{o}}(X)$. Then the gauges induced by the dual bodies $K^{\omega_1}$ and $K^{\omega_2}$ are isometric. 
\end{teo}
\begin{proof} Denote by $\gamma_1$ and $\gamma_2$ the gauges whose unit balls are $K^{\omega_1}$ and $K^{\omega_2}$, respectively. Let $\mathcal{I}_1:X^*\rightarrow X$ and $\mathcal{I}_2:X^*\rightarrow X$ be the identifications given by $\omega_1$ and $\omega_2$, respectively, as in Section \ref{ident}. We claim that the (linear) map
\begin{align*} T := \mathcal{I}_2\circ\mathcal{I}_1^{-1}:(X,\gamma_1)\rightarrow (X,\gamma_2)
\end{align*}
is an isometry. Of course, to see this we have to prove that $T$ is gauge-preserving. Indeed, if $x \in K^{\omega_1}$, then $\mathcal{I}_1^{-1}(x) \in K^{\circ}$. Hence $\mathcal{I}_2\circ\mathcal{I}_1^{-1}(x) \in K^{\omega_2}$. This gives that $T(K^{\omega_1})\subseteq K^{\omega_2}$. On the other hand, if $x \in K^{\omega_2}$, then $x = \mathcal{I}_2(f)$ for some $f \in K^{\circ}$, and $f = \mathcal{I}_1^{-1}(z)$ for some $z \in K^{\omega_1}$. It follows that $x = T(z)$, and this gives the inclusion $K^{\omega_2} \subseteq T(K^{\omega_1})$. From Lemma \ref{liniso}, the equality $T(K^{\omega_1}) = K^{\omega_2}$ yields that $T$ is an isometry.

\end{proof}

\begin{remark} Notice that both gauge spaces $(X,\gamma_1)$ and $(X,\gamma_2)$ are linearly isometric to the polar gauge space $(X^*,\gamma^*)$ by the maps $\mathcal{I}_1$ and $\mathcal{I}_2$, respectively. This justifies the construction of the isometry in the proof of the theorem. The result can also be seen as a consequence of the proof of the second part of Theorem \ref{gaugeiso}, setting $T = \mathrm{id}_X$.
\end{remark}

Let $Y$ be a subspace of the symplectic vector space $(X,\omega)$. We define its \emph{symplectic complement} to be 
\begin{align*} Y^{\perp} := \{x \in X:\omega(x,y) = 0 \ \mathrm{for \ all} \ y \in Y\}.
\end{align*}
For more on symplectic linear algebra we refer the reader to \cite[Chapter 2]{mcduff}. As we will see now, the symplectic complement of a given direction plays an important role in characterizing where the dual norm is attained. In what follows, for each non-zero vector $x \in X$ we denote by
\begin{align*} \{x\}^{\perp} := \{z \in X:\omega(x,z) = 0\}
\end{align*}
the symplectic complement of $\mathrm{span}\{x\}$, which is clearly a hyperplane containing $x$. The next technical lemma states that each hyperplane is somehow related to one of its \emph{directions} (which we define as one-dimensional vector subspaces) by symplectic complementation.

\begin{lemma} Let $H \subseteq (X,\omega)$ be a hyperplane. Then there exists precisely one direction $\mathrm{span}\{x\} \in H$ such that $H = \{x\}^{\perp}$.  
\end{lemma}
\begin{proof} If $H$ is not the symplectic complement of any of its elements, then for any $x \in H$ there exists $z \in H$ such that $\omega(x,z) \neq 0$. It follows that the restriction $\omega|_H$ is non-degenerate, from where $\omega|_H$ is a symplectic form. This is a contradiction, since $H$ is odd-dimensional. 

Assume now that $H = \{x\}^{\perp} = \{y\}^{\perp}$. Since the non-degenerate functionals $f(\cdot)= \omega(x,\cdot)$ and $g(\cdot) = \omega(y,\cdot)$ have the same kernel, we see that $f = \alpha g$ for some $\alpha \neq 0$. It follows immediately that $x = \alpha y$.

\end{proof}

\begin{remark} If $X = \mathbb{R}^{2n}$ is equipped with the standard symplectic form $\omega_0:= \sum dx_j\wedge dy_j$ and with the standard inner product $\langle\cdot,\cdot\rangle$, then for any hyperplane $H$ we have that $H = \{J_0x\}^{\perp}$, where $x$ is the unique orthogonal direction to $H$ and 
\begin{align*} J_0 = \left(\begin{matrix}0_n & -\mathrm{id}_n \\ \mathrm{id}_n & 0_n \end{matrix}\right)\!:\mathbb{R}^{2n}\rightarrow\mathbb{R}^{2n}
\end{align*}
is the standard complex structure on $\mathbb{R}^{2n}$. However, notice that the identification of a given hyperplane $H$ with a direction $x \in H$ such that $H = \{x\}^{\perp}$ depends only on the fixed symplectic form on $X$. For a given hypersurface, this identification gives birth to the \emph{characteristic foliation} (which can also be obtained from Hamiltonian flows, see \cite[Chapter 1]{mcduff}). We emphasize that this is a method for obtaining the characteristic foliation which is not only independent on a Hamiltonian function, but does also not depend on a fixed inner product in $\mathbb{R}^{2n}$. In Section \ref{planar} the periodic leaves of the characteristic foliation (the \emph{closed characteristics}) will be discussed in more detail.
\end{remark}

The following is an extension of Proposition \ref{ortmax} for higher dimensional gauge spaces. Recall that we denote the unit ball of a gauge $\gamma$ by $K$ or $B_{\gamma}$, and that $\gamma_{\omega}$ stands for the dual gauge given by a fixed symplectic form $\omega$.
\begin{prop}\label{ortmax2} Let $(X,\gamma)$ be a gauge space with a fixed symplectic form $\omega$. Given an arbitrary non-zero vector $x \in X$, we have that
\begin{align*} \gamma_{\omega}(x) = \omega(x,y_0)
\end{align*}
for some $y_0 \in \partial B_{\gamma}$ if and only if $y_0 \dashv \{x\}^{\perp}$ and $\omega(x,y_0) > 0$. 
\end{prop}
\begin{proof} The hyperplane $\{x\}^{\perp}$ splits $X$ into two half-spaces, one of them being
\begin{align*} X^+_x := \{z \in X:\omega(x,z) \geq 0\},
\end{align*}
which we call the \emph{positive half-space}. We denote by $\partial^+_xB_{\gamma}$ the portion of the unit sphere contained in the interior of the positive half-space. 

Geometrically, saying that $y_0 \dashv \{x\}^{\perp}$ and $\omega(x,y_0) > 0$ is equivalent to stating that $\{x\}^{\perp}$ supports $B_{\gamma}$ at $y_0 \in \partial^+_xB_{\gamma}$. In this case we may write
\begin{align*} \partial^+_xB_{\gamma} = \left\{\!\frac{y_0+tz}{\gamma(y_0+tz)}:z \in \{x\}^{\perp} \ \mathrm{and} \ t \in \mathbb{R}\!\right\}\!. 
\end{align*}
Since the maximum of $\omega(x,\cdot)$ over $B_{\gamma}$ is clearly attained for a point of $\partial^+_xB_{\gamma}$, and
\begin{align*} \omega\!\left(\!x,\frac{y_0+tz}{\gamma(y_0+tz)}\!\right)\! = \frac{\omega(x,y_0)}{\gamma(y_0+tz)} \leq \omega(x,y_0)
\end{align*}
for any $z\in \{x\}^{\perp}$ and every $t \in \mathbb{R}$, it follows that $\gamma_{\omega}(x) = \omega(x,y_0)$. Notice that the last inequality comes from the fact that $\gamma(y_0 + tz) \geq 1$. 

If $\gamma_{\omega}(x) = \omega(x,y_0)$ for a point $y_0 \in \partial B_{\gamma}$, then it is clear that $\omega(x,y_0) > 0$. However, assume that there exists a (non-zero) vector $z \in \{x\}^{\perp}$ such that $y_0$ is not left-orthogonal to $z$. In this case, there exists a number $t_1 \in \mathbb{R}$ such that $y_0+t_1z \in \mathrm{int}(B_{\gamma})$, that is, $\gamma(y_0+t_1z) < 1$. This gives
\begin{align*} \omega\!\left(\!x,\frac{y_0+t_1z}{\gamma(y_0+t_1z)}\!\right)\! = \frac{\omega(x,y_0)}{\gamma(y_0+t_1z)} > \omega(x,y_0) = \gamma_{\omega}(x),
\end{align*}
which is a contradiction. 

\end{proof}

As in the two-dimensional case, the characterization of the points where the dual gauge is attained yields precisely the tool that we need to understand how orthogonality behaves under duality. 

\begin{teo} Let $x,y \in X$ be non-zero vectors. If $y \dashv_{\omega} \{x\}^{\perp}$ and $\omega(x,y) > 0$, then $-x \dashv \{y\}^{\perp}$.  
\end{teo}
\begin{proof} Without loss of generality, assume that $y \in K^{\omega}$. According to Proposition \ref{ortmax2}, we have that if $y \dashv_{\omega} z$ for any $z \in \{x\}^{\perp}$ and $\omega(x,y) > 0$, then
\begin{align*} (\gamma_{\omega})_{\omega}(x) = \omega(x,y).
\end{align*}
Recalling that $(\gamma_{\omega})_{\omega}(x) = \gamma_{-K}(x) = \gamma(-x)$, we may write the equality above as
\begin{align*} 1 = \gamma_{\omega}(y) = \omega\!\left(\!y,\frac{-x}{\gamma(-x)}\!\right)\!.
\end{align*}
Using Proposition \ref{ortmax2} again, as well as positive left-homogeneity, we get immediately that the equality above gives $-x \dashv \{y\}^{\perp}$. 

\end{proof}

\begin{coro}\label{orthhigh} Let $x,y \in X$ be non-zero vectors such that $x \dashv \{y\}^{\perp}$ and $\omega(y,x) > 0$. Then $y \dashv_{\omega}\{x\}^{\perp}$. 
\end{coro}
\begin{proof} If $x \dashv \{y\}^{\perp}$, then from the (skew-)duality $(K^{\omega})^{\omega} = -K$ and from Lemma \ref{oppgauge} we get $-x (\dashv_{\omega})_{\omega} \{-y\}^{\perp}$. Since $\omega(-y,-x) = \omega(y,x) >0$, we get from the theorem above that $-(-y)\dashv_{\omega}\{-x\}^{\perp}$. Hence $y \dashv_{\omega}\{x\}^{\perp}$. 

\end{proof}

\begin{remark} Notice that if $\mathrm{dim}(X) > 2$, then $x \dashv y$ and $\omega(y,x) > 0$ does not necessarily imply $y \dashv_{\omega} x$. 
\end{remark}

\section{Planar sections and characteristics} \label{planar}

A subspace $Y \subseteq (X,\omega)$ is called \emph{symplectic} if the restriction $\omega|_Y$ is a symplectic form on $Y$. Equivalently, $Y$ is symplectic if and only if its symplectic complement $Y^{\perp}$ is such that $Y\cap Y^{\perp} = \{0_X\}$. A two-dimensional symplectic subspace will be called a \emph{symplectic plane}. Of course, a plane $Y\subseteq X$ is symplectic if and only if there exist vectors $x,y \in Y$ such that $\omega(x,y) \neq 0$. 

\begin{lemma} If $(X,\omega)$ is a $(2n)$-dimensional symplectic vector space, then $X$ can be written as the direct sum
\begin{align*} X = \bigoplus_{j=1}^nY_j
\end{align*}
of $n$ symplectic planes $Y_j$ such that $Y_i\cap Y_j = \{0_X\}$ whenever $i\neq j$. If $n = 1$, then the decomposition is trivial. 
\end{lemma}
\begin{proof} Let $(x_1,\ldots,x_n,y_1,\ldots,y_n)$ be a \emph{symplectic basis} of $X$, that is, a basis such that $\omega(x_i,x_j) = \omega(y_i,y_j) = 0$ and $\omega(x_i,y_j) = \delta_{ij}$ for any $i,j = 1,\ldots, n$ (for the existence of such a basis see \cite[Theorem 2.1.3]{mcduff}). Now simply put
\begin{align*} Y_j = \mathrm{span}\{x_j,y_j\}.
\end{align*}
It is clear that each $Y_j$ is symplectic, since $\omega(x_j,y_j) \neq 0$, and that $X = \bigoplus_{j=1}^nY_j$.

\end{proof}

The idea is that, when restricted to such a subspace $Y\subseteq X$, the ambient gauge and symplectic form yield a planar gauge geometry and a symplectic form on $Y$, respectively. Hence our next step is to understand the behavior of these planar geometries under duality. 

\begin{teo} Let $(X,\gamma,\omega)$ be an even-dimensional vector space endowed with a gauge $\gamma$ and a symplectic form $\omega$, and let $Y\subseteq X$ be a symplectic plane. Then we have
\begin{align*} (K\cap Y)^{\omega} = \mathrm{proj}_Y(K^{\omega}),
\end{align*}
where $(K\cap Y)^{\omega}$ is the image of the polar body
\begin{align*} (K\cap Y)^{\circ} := \{f:Y\rightarrow\mathbb{R}:f(y) \leq 1 \ \mathrm{for \ all} \ y \in K\cap Y\}
\end{align*}
under the identification of $Y$ and $Y^*$ (as constructed in Section \ref{ident}) given by $\omega|_Y$, and $\mathrm{proj}_Y(K^{\omega})$ is the projection with respect to the direct sum $X = Y\oplus Y^{\perp}$. 
\end{teo}
\begin{proof} Let $y_0 \in \mathrm{proj}_Y(K^{\omega})$. Then we may write $z = y_0 + x$, for some $z\in K^{\omega}$ and some $x \in Y^{\perp}$. For any $y \in K\cap Y$, we have that
\begin{align*} \omega(y_0,y) = \omega(z-x,y) = \omega(z,y) \leq 1,
\end{align*}
because $z \in K^{\omega}$ and $y \in K$. Thus $y_0 \in (K\cap Y)^{\omega}$, and it follows that $\mathrm{proj}_Y(K^{\omega})\subseteq (K\cap Y)^{\omega}$. 

To prove the converse inclusion, let $y_0 \in \partial(K\cap Y)^{\omega}$. Hence we have that $\omega(y_0,y) \leq 1$ for any $y \in K\cap Y$. If $y_0 \notin \mathrm{proj}_Y(K^{\omega})$, then we may write $y_0 = \alpha y_1$ for some $\alpha > 1$ and $y_1 \in \partial(\mathrm{proj}_Y(K^{\omega}))$. Let $x_0 \in K\cap Y$ be such that $\omega(y_0,x_0) = 1$ (such a point exists because $y_0$ is in the boundary of the dual of $K\cap Y$). Since $y_1 \in \partial(\mathrm{proj}_Y(K^{\omega})) \subseteq
 \mathrm{proj}_Y(\partial K^{\omega})$, we may write
\begin{align*} z = y_1 + x
\end{align*}
for some $z \in \partial K^{\omega}$ and some $x \in Y^{\perp}$. From this we get
\begin{align*} \omega(z,x_0) = \omega(y_1+x,x_0) = \omega(y_1,x_0) = \alpha\omega(y_0,x_0) = \alpha > 1,
\end{align*}
which is a contradiction because $z \in K^{\omega}$ and $x_0 \in K$. It follows that $\partial (K\cap Y)^{\omega} \subseteq \partial(\mathrm{proj}_Y(K^{\omega}))$. This, together with the inclusion proved previously, gives the desired. 

\end{proof}

The reader may carefully notice that we do not have $(K\cap Y)^{\omega} = K^{\omega}\cap Y$ for a given plane $Y$. Actually, we will see that this equality guarantees the existence of a planar closed characteristic on $\partial K$, in the case when $K$ is smooth. First, define $J:\partial K\rightarrow \partial K^{\omega}$ as the map such that, for each $x \in\partial K$, $Jx$ is the (unique) vector with the property that $\{Jx\}^{\perp}$ supports $K$ at $x$, and $\omega(Jx,x) = 1$ (observe that from Proposition \ref{ineqdual} we have $\gamma_{\omega}(Jx) = 1$). The characteristics of $\partial K$ are precisely the curves which are tangent to the direction $Jx$ at each $x \in \partial K$. Next we state a technical lemma to understand how the map $J$ behaves under duality. 

\begin{lemma} Assume that $K$ is smooth and strictly convex. Let $J^{\omega}:\partial K^{\omega}\rightarrow \partial(-K)$ be the map $J$ as defined above regarding the dual gauge. Then we have $J^{\omega}\circ J = -\mathrm{id}_{\partial K}$. 
\end{lemma}
\begin{proof} For a given $y \in \partial K^{\omega}$, $J^{\omega}y$ is the unique vector such that $\{J^{\omega}\}^{\perp}$ supports $K^{\omega}$ at $y$, and $\omega(J^{\omega}y,y) = 1$. If $y = Jx$, then from Corollary \ref{orthhigh} we have that $J^{\omega}y$ points in the direction of $x$. Finally, the equality $1 = \omega(J^{\omega}y,y) = \omega(\alpha x,Jx)$ gives $\alpha = -1$. Hence $J^{\omega}(Jx) = -x$. 

\end{proof}

Now we arrive at the mentioned characterization of planar closed characteristics of $\partial K$. Recall that the integral curves of the field of directions given by $Jx$, $x \in \partial K$, are the \emph{characteristics} of $\partial K$. A characteristic $c$ is \emph{closed} if it is a \emph{closed curve}, meaning that it admits a parametrization $c(t):\mathbb{S}^1\rightarrow \partial K$. To avoid any confusion, recall also that by \emph{plane} we mean a two-dimensional vector subspace. Notice that the origin can always be translated such that a given affine plane passes through the origin. 

\begin{teo} Let $(X,\omega)$ be a symplectic vector space, and let $K \subseteq X$ be a smooth convex body containing the origin in its interior. Then, for a given symplectic plane $Y\subseteq X$, we have that $(K\cap Y)^{\omega} = K^{\omega}\cap Y$ if and only if $\partial K\cap Y$ is a planar closed characteristic of $\partial K$. 
\end{teo}
\begin{proof} Assume first that $(K\cap Y)^{\omega} = K^{\omega}\cap Y$. Let $y$ be a point of the boudary of the planar convex body $K\cap Y$. We let $z \in Y$ be such that the line in the direction of $z$ supports $K\cap Y$ at $y$, and $\omega(z,y) = 1$. Denote by $H$ the hyperplane which supports $K$ at $y$, and notice that we have $z \in H$. We will prove that $H = \{z\}^{\perp}$. Indeed, since it is clear that $z \in (K\cap Y)^{\omega}$, we have from the hypothesis that $z \in K^{\omega}$. If $\omega(z,x) \neq 0$ for some $x \in H$, then we may assume that $\omega(z,x) > 0$, yielding
\begin{align*} \omega\!\left(z,\frac{y+tx}{\gamma(y+tx)}\right) = 1 + \frac{t\omega(z,x)}{\gamma(y+tx)} > 1
\end{align*}
provided $t > 0$. This is in contradiction to the fact that $z \in K^{\omega}$. It follows that $H = \{z\}^{\perp}$. This, with the equality $\omega(z,y) = 1$, gives $z = Jy$. It follows that the curve $\partial K\cap Y$ is tangent at each point to the characteristic direction of $K$ at that point, and hence it is a (closed) characteristic of $\partial K$.

Now we prove the converse. For simplicity, we write $K\cap Y = K_Y$. This planar convex body induces a gauge, which is simply the restriction of the ambient gauge of $X$ to $Y$. The restriction of the symplectic form to $Y$ (recall that $Y$ is a symplectic plane) yields the dual gauge body $(K\cap Y)^{\omega} = K_Y^{\omega}$. If $\partial K\cap Y$ is a closed characteristic, then $Jy \in Y$ for any $y \in \partial K_Y$. Since $Jy$ supports $K_Y$ at $y$ and $\omega(Jy,y) = 1$, we get that $Jy \in \partial K_Y^{\omega}$. Hence we may write
\begin{align*}(K\cap Y)^{\omega} = \{\alpha Jy:y \in \partial K\cap Y \ \mathrm{and} \ 0 \leq \alpha \leq 1\}.
\end{align*}

On the other hand, $J$ maps $\partial K$ onto $\partial K^{\omega}$, and under the hypothesis we get that $J$ maps $\partial K\cap Y$ onto $\partial K^{\omega}\cap Y$. Since $\partial K^{\omega}\cap Y = \partial(K^{\omega}\cap Y)$, we get that $(K\cap Y)^{\omega} = K^{\omega}\cap Y$.  

\end{proof}

\begin{remark} Notice that if $Y$ is a given plane and $\partial K\cap Y$ is a characteristic of $\partial K$, then it is easy to see that $Y$ is symplectic. If an affine planar section of $K$ determines a characteristic at the boundary, then the theorem still holds, up to a translation of $K$. 
\end{remark}

To finish this section, we show that if $K$ is a smooth convex body, then the closed characteristics of $\partial K$ can be characterized in terms of the dual gauge of $K$. Before this, let us fix some concepts. A differentiable curve $c(t):[a,b]\rightarrow(X,\gamma)$ is \emph{regular} if $c'(t) \neq 0$ for each $t \in [a,b]$, and it is said to be \emph{simple} if it has no self-intersections. The \emph{length} of $c(t)$ in the gauge geometry given by $\gamma$ is the number
\begin{align*} L_{\gamma}(c) = \int_a^b\gamma(c'(t))\, dt,
\end{align*}
which, in general, depends (only) on the orientation of the parametrization. If $c(t):\mathbb{S}^1\rightarrow\partial K$ is a closed characteristic, then $c'(t)$ is a non-zero vector in the direction of $Jc(t)$, from where we get that $\omega(c'(t),c(t)) \neq 0$ for each $t \in \mathbb{S}^1$. Hence, up to a reparametrization, we may assume that $\omega(c'(t),c(t)) > 0$ for any $t \in \mathbb{S}^1$. A differentiable, regular curve for which the latter holds is said to be \emph{positively parametrized}. Finally, we define the \emph{symplectic area} of a closed curve $c$ (positively parametrized, say) to be
\begin{align*} A(c) = \frac{1}{2}\int_{\mathbb{S}^1}\omega(c'(t),c(t))\, dt.
\end{align*}
Of course, if $c$ is contained in a symplectic plane, then $A(c)$ is the usual area of the region which it encloses, in the area element given by $\omega$.

\begin{teo} Let $c(t):\mathbb{S}^1\rightarrow\partial K$ be a differentiable, regular, simple curve which is positively parametrized. Then $c$ is a closed characteristic of $\partial K$ if and only if
\begin{align*} 2A(c) = L_{\omega}(c),
\end{align*}
where $L_{\omega}(c)$ denotes the length of $c$ in the geometry given by the dual gauge $\gamma_{\omega}$ of $\gamma = \gamma_K$. 
\end{teo}
\begin{proof} If $c(t):\mathbb{S}^1\rightarrow\partial K$ is a closed characteristic, then for each $t \in \mathbb{S}^1$ we have that $c'(t)$ points in the direction of $Jc(t)$. It follows from Proposition \ref{ortmax2} that $\omega(c'(t),c(t)) = \gamma_{\omega}(c'(t))$. Thus,
\begin{align*} 2A(c) = \int_{\mathbb{S}^1}\omega(c'(t),c(t))\,dt = \int_{\mathbb{S}^1}\gamma_{\omega}(c'(t))\,dt = L_{\omega}(c).
\end{align*}
For the converse, notice that we always have $\omega(c'(t),c(t)) \leq \gamma_{\omega}(c'(t))$. Again from Proposition \ref{ortmax2} (and continuity) we have that if $c$ is not a closed characteristic, then the inequality is strict in some interval $J\subseteq\mathbb{S}^1$. It follows from integration that $2A(c) < L_{\omega}(c)$. 

\end{proof}

As a scholium, we have that the inequality $2A(c) \leq L_{\omega}(c)$ holds for any differentiable, regular, simple curve $c(t):\mathbb{S}^1\rightarrow\partial K$ which can be positively parametrized (notice that this condition can be dropped if $K$ is centrally symmetric). This can be thought of as a sort of isoperimetric inequality (for curves on $\partial K$), where the optimal cases are given precisely by the closed characteristics of $\partial K$. 

In the planar case, we obviously have that $2A(K) = L_{\omega}(\partial K)$, where $K$ is the unit disk of the gauge $\gamma$, and $\partial K$ is positively parametrized (this duality has been observed already for the symmetric case; see, e.g., \cite{Mar-Swa}). This is in line with the fact that, when $(X,\gamma_K)$ is a gauge plane, the unique closed characteristic of $\partial K$ is the curve $\partial K$ itself. In higher dimensions, the closed characteristics of $\partial K$ can be interpreted as a sort of curves for which this ``duality principle" is preserved. 

\begin{remark} In the class of smooth convex bodies, the \emph{Ekeland-Hofer capacity} and the \emph{Hofer-Zehnder capacity} coincide, and both are equal to the minimum symplectic area of the closed characteristics (see \cite[Theorem 1.3]{Art-Ost2}, and also \cite{Art-Ost}). From the theorem above we have that this can be interpreted in terms of the metric: the capacity is half the minimum length among the closed characteristics (positively oriented) measured in the dual gauge. 
\end{remark}

\end{document}